\documentclass[a4paper,11pt]{article}
\usepackage{amsmath}
\usepackage{amssymb}
\usepackage{theorem}
\usepackage{pstricks}
\usepackage{euscript}
\usepackage{epic,eepic}
\usepackage{graphicx}
\PassOptionsToPackage{normalem}{ulem}
\usepackage{amsmath}
\usepackage{amssymb}
\usepackage{theorem}
\usepackage{pstricks}
\usepackage{euscript}
\usepackage{epic,eepic}
\usepackage{graphicx}
\usepackage{setspace}
\PassOptionsToPackage{normalem}{ulem}

\newcommand{\menge}[2]{\big\{{#1} \;|\; {#2}\big\}}

\newcommand{\emp}{\ensuremath{{\varnothing}}}

\newcommand{\scal}[2]{\left\langle{#1}\mid {#2} \right\rangle} 
\newcommand{\pscal}[2]{\langle\langle{#1}\mid{#2}\rangle\rangle} 

\newcommand{\vuo}{\ensuremath{\mbox{\footnotesize$\square$}}}

\newcommand{\HH}{\ensuremath{\mathcal H}}
\newcommand{\GG}{\ensuremath{\mathcal G}}

\newcommand{\RR}{\ensuremath{\mathbb R}}
\newcommand{\KK}{\ensuremath{\mathcal K}}

\newcommand{\NN}{\ensuremath{\mathbb N}}

\newcommand{\dom}{\ensuremath{\operatorname{dom}}}

\newcommand{\prox}{\ensuremath{\operatorname{prox}}}

\newcommand{\argmin}{\ensuremath{\operatorname{argmin}}}
\newcommand{\ran}{\ensuremath{\operatorname{ran}}}
\newcommand{\zer}{\ensuremath{\operatorname{zer}}}
\newcommand{\gra}{\ensuremath{\operatorname{gra}}}

\newcommand{\vv}{\ensuremath{\boldsymbol{v}}}

\newcommand{\xx}{\ensuremath{\boldsymbol{x}}}

\newcommand{\ww}{\ensuremath{\boldsymbol{w}}}
\newcommand{\BB}{\ensuremath{\boldsymbol{B}}}

\newcommand{\AAA}{\ensuremath{\boldsymbol{A}}}
\newcommand{\BBB}{\ensuremath{\boldsymbol{B}}}

\newcommand{\Id}{\ensuremath{\operatorname{Id}}}

\newcommand{\weakly}{\ensuremath{\rightharpoonup}}

\newtheorem{theorem}{Theorem}[section]
\newtheorem{lemma}[theorem]{Lemma}

\newtheorem{corollary}[theorem]{Corollary}

\theoremstyle{plain}{\theorembodyfont{\rmfamily}
}
\theoremstyle{plain}{\theorembodyfont{\rmfamily}
}
\theoremstyle{plain}{\theorembodyfont{\rmfamily}
}
\theoremstyle{plain}{\theorembodyfont{\rmfamily}
\newtheorem{example}[theorem]{Example}}
\theoremstyle{plain}{\theorembodyfont{\rmfamily}
\newtheorem{problem}[theorem]{Problem}}
\theoremstyle{plain}{\theorembodyfont{\rmfamily}
\newtheorem{remark}[theorem]{Remark}}
\theoremstyle{plain}{\theorembodyfont{\rmfamily}
}

\definecolor{labelkey}{rgb}{0,0.08,0.45}
\definecolor{refkey}{rgb}{0,0.6,0.0}
\definecolor{Brown}{rgb}{0.45,0.0,0.05}
\definecolor{dgreen}{rgb}{0.00,0.49,0.00}
\definecolor{dblue}{rgb}{0,0.08,0.75}
\numberwithin{equation}{section}

\title{A reflected forward-backward splitting method for monotone inclusions involving Lipschitzian operators}
\author{\   Volkan Cevher and 
B$\grave{\text{\u{a}}}$ng C\^ong V\~u$^*$\\[5mm]
 \\
Laboratory for Information and Inference Systems, EPFL, Lausanne, Switzerland\\
 volkan.cevher@epfl.ch; bang.vu@epfl.ch or bangcvvn@gmail.com} 
\begin{document}
\maketitle
\begin{abstract}
The proximal extrapolated gradient method \cite{Malitsky18a} is an extension of the projected reflected gradient method  \cite{Malitsky15}.
Both methods were proposed  for solving the classic variational inequalities. In this paper,
we investigate the projected reflected gradient method, in the general setting, for solving monotone inclusions involving Lipschitzian operators. As a result, we obtain a simple method
for finding  a zero point of the sum of two monotone operators  where one of them is Lipschizian. We also show that one can improve the range of the stepsize 
of this method for the case when the Lipschitzian operator is restricted to be cocoercive. A nice combination of this method and the forward-backward splitting was proposed. As a result,
we obtain a new splitting method for finding a zero point of the sum of three operators ( maximally monotone +  monotone Lipschitzian + cocoercive).  Application to
composite monotone inclusions are demonstrated.
\end{abstract}


\section{Introduction}
The forward-backward-forward splitting method (FBFS) or Tseng's splitting method was firstly appeared  in \cite{Tseng00}.  This method was proposed to  find  a zero point of the sum of two monotone operators acting on a real Hilbert space $(\HH,\scal{\cdot}{\cdot})$, namely,
\begin{equation}
\label{e:prob1}
\text{find $\overline{x} \in\HH$ such that}\;
0 \in A\overline{x}+B\overline{x}.
\end{equation}
under the assumption that $A\colon\HH\to 2^{\HH}$ is a maximally monotone, $B:\HH\to\HH$ 
be a monotone and $\mu$-Lipschitzian, i.e., 
\begin{equation}
(\forall x\in\HH)(\forall y\in\HH)\; \|Bx-By\| \leq \mu \|x-y\|,
\end{equation}
and that such a solution exists. The FBFS method operates according to the routine
\begin{equation}\label{e:Tseng}
\begin{cases}
y_n = x_n-\gamma B x_n\\
z_n = (\Id+\gamma A)^{-1}y_n\\ 
r_n = z_n-\gamma Bz_n\\
x_{n+1} = x_n+ r_n-y_n.
\end{cases}
\end{equation}
The weak convergence of $(x_n)_{n\in\NN}$ to a solution of \eqref{e:prob1} was proved under the condition $0 < \gamma < 1/\mu$.
\noindent Inexact version of the FBF method was investigated in \cite{siop2}.
 Then, variable metric version and the stochastic version the FBF method are in \cite{Bang13} and \cite{Bang15}, respectively.
One of the most important example of  $B$ is the case when $B$  is a linear skew operator \cite{siop2} 
where monotone plus skew model plays a central role in solving  primal-dual monotone inclusions  and primal-dual convex 
optimization problems. The main idea of \cite{siop2} was then developed rapidly in \cite{plc6}. Several developments and extensions 
of \cite{plc6} are in 
 \cite{Combettes13,Bot2014,Bang14,Bang13}.
  
 \noindent The advantage of this framework is its generality and the main disadvantage of \eqref{e:Tseng} is that it requires two calls of $B$ per one iteration.
This issue was recently resolved in \cite{Malitsky18b}. Specifically, they propose a forward-reflected-backward splitting  method (FRBS) for solving \eqref{e:prob1}, namely,
\begin{equation}\label{e:FRB}
\gamma \in \left]0,+\infty \right[, \quad  x_{n+1} = (\Id +\gamma A)^{-1}(x_n-2\gamma Bx_n+ \gamma Bx_{n-1}).
\end{equation}
The weak convergence of the iterates generated by \eqref{e:FRB} is proved under the condition $\gamma \in \left]0,1/(2\mu)\right[$.
If $B$ is linear and $A$  is the normal cone of some non-empty closed convex set $K$,
the FRBS method  admits the same structure as the reflected 
projected gradient methods for variational inequalities  \cite{Malitsky15}, namely,
\begin{equation} \label{e:Malitsky}
\gamma \in \left]0,+\infty \right[,\quad  x_{n+1} = (\Id +\gamma N_K)^{-1}(x_n- \gamma B(2x_n-x_{n-1})).
\end{equation}
For any $\mu$-Lipschitzian monotone operator $B$, the weak convergence of the iterates generated by \eqref{e:Malitsky} 
is proved under the condition $\gamma \in \left]0,(\sqrt{2}-1)/\mu\right[$.  When $N_K$ is replaced by a subdifferential 
of some proper lower semicontinuous convex function $f$,
line-search versions \eqref{e:Malitsky} are proposed in \cite{Malitsky18a}.

The objective of this paper is two-folds. We firstly investigate the convergence of \eqref{e:Malitsky} for \eqref{e:prob1} for any maximally monotone operator $A$, i.e., we 
propose to investigate the convergence of the following reflected forward-backward splitting method (RFBS) for  \eqref{e:prob1}:
\begin{equation}\label{e:reflected}
\begin{cases}
y_n= 2x_{n}-x_{n-1}\\
\gamma \in \left]0,+\infty \right[\\
x_{n+1} = (\Id+\gamma A)^{-1}(x_n- \gamma By_n).
\end{cases}
\end{equation}
Secondly,  we investigate the problem of  improving the range of the stepsize $\gamma$ for the cases where $B$ is cocoercive operator. 

In Section \ref{e:weak}, we prove the weak convergence of \eqref{e:reflected} and exploit the cocoercivity of $B$ to improve the range of the stepsize. 
We propose a combination of \eqref{e:reflected} and the standard forward-backward splitting and prove its convergence in Section \ref{e:SRFB}. 
The last Section is an application to composite monotone inclusions involving the parallel sums and Lipschitzian monotone operators.

\noindent{\bf Notations.} (See  \cite{livre1}) 
 The scalar products and the  associated 
norms of all Hilbert spaces used in this paper are denoted respectively by 
$\scal{\cdot}{\cdot}$ and $\|\cdot\|$. 
We denote by $\mathcal{B}(\HH,\GG)$
the space of all bounded linear operators from $\HH$ to $\GG$. 
The symbols $\weakly $ and $\to$ denote respectively 
weak and strong convergence.
Let $A\colon\HH\to 2^{\HH}$ be a set-valued operator. The domain of $A$ is denoted by $\dom(A)$ that 
is a set of all $x\in\HH$ such that $Ax\not= \emp$. The range of $A$ is $\ran(A) = \menge{u\in\HH}{(\exists x\in\HH) u\in Ax }$. 
The graph of $A$ is 
$\gra(A) = \menge{(x,u)\in\HH\times\HH}{u\in Ax}$. 
The inverse of $A$ 
is $A^{-1}\colon u \mapsto \menge{x}{ u\in Ax}$. 
The zero set of $A$ is $\zer(A) = A^{-1}0$.
We say that $A$ is  monotone if 
\begin{equation}\label{oioi}
\big(\forall u\in Ax\big)\big(\forall (y,v)\in\gra A\big)
\quad\scal{x-y}{u-v}\geq 0,
\end{equation}
and it is maximally monotone if there exists no monotone operator $B$ such that $\gra(B)$ properly contains $\gra(A)$.
The resolvent of $A$ is
\begin{equation}
 J_A=(\Id + A)^{-1}, 
\end{equation}
where $\Id$ denotes the identity operator on $\HH$. 
A single-valued operator $B\colon\HH\to\HH$ is $\beta$-cocoercive, for some $\beta\in \left]0,+\infty\right[$, if
\begin{equation}
(\forall (x,y)\in\HH^2)\; \scal{x-y}{Bx-By} \geq \beta\|Bx-By\|^2.
\end{equation}  
The parallel sum of two operator $A\colon\HH\to 2^{\HH}$ and $B\colon\HH\to 2^{\HH}$ is $A\vuo B = (A^{-1}+B^{-1})^{-1}$.
 The class of all lower semicontinuous convex functions 
$f\colon\HH\to\left]-\infty,+\infty\right]$ such 
that $\dom f=\menge{x\in\HH}{f(x) < +\infty}\neq\emp$ 
is denoted by $\Gamma_0(\HH)$. Now, let $f\in\Gamma_0(\HH)$.
%The conjugate of $f$ is the function $f^*\in\Gamma_0(\HH)$ defined by
%$f^*\colon u\mapsto
%\sup_{x\in\HH}(\scal{x}{u} - f(x))$, and 
The subdifferential 
of $f\in\Gamma_0(\HH)$ is the maximally monotone operator 
\begin{equation}
 \partial f\colon\HH\to 2^{\HH}\colon x
\mapsto\menge{u\in\HH}{(\forall y\in\HH)\quad
\scal{y-x}{u} + f(x) \leq f(y)}
\end{equation} 
%with inverse given by
%\begin{equation}
%(\partial f)^{-1}=\partial f^*.
%\end{equation}
Moreover,
the proximity operator of $f$ is
\begin{equation}
\label{e:prox}
\prox_f=J_{\partial f} \colon\HH\to\HH\colon x
\mapsto\underset{y\in\HH}{\argmin}\: f(y) + \frac12\|x-y\|^2.
\end{equation}
Various closed-form expressions of the proximity operators are in \cite[Chapter 24]{livre1}.
\section{Weak convergence}
\label{e:weak}
In this section, we prove the weak convergence of RFBS in \eqref{e:reflected}  under the condition $\gamma \in \left]0, (\sqrt{2}-1)/\mu \right[$. 
Moreover, we also exploit the cocoercivity of $B$ to improve the range of the stepsize $\gamma$.
Throughout this section, $(x_n)_{n\in\NN}$ and $(y_n)_{n\in\NN}$ are generated by \eqref{e:reflected}.

\begin{theorem}\label{t:1} The following hold.
\begin{enumerate}
\item\label{t:1ii} 
Suppose that $B$ is $\beta$-cocoercive and $ \gamma \in \left]0, \beta(1-\epsilon)/2 \right]$, for some $\epsilon \in \left]0,1\right[$,
then $x_n \weakly \overline{x} \in\zer(A+B)$.
\item \label{t:1i} Suppose that  $\gamma \in \left]0, (\sqrt{2}-1)/\mu \right[$,  then $x_n \weakly \overline{x} \in\zer(A+B)$.
\end{enumerate}
\end{theorem}
\begin{proof} Let $x\in\zer(A+B)$.\\
\ref{t:1ii}: Suppose that $B$ is $\beta$-cocoercive.
For any $n\in\NN$, set $p_{n+1}= x_n- \gamma By_n -x_{n+1}$. Then $p_{n+1} \in \gamma Ax_{n+1}$
We have 
\begin{equation}
\scal{x_{n}+\gamma By_{n-1} -x_{n-1} }{x_{n+1}-y_n} = - \scal{p_{n}}{x_{n+1}-y_n}.
\end{equation}
Note that $x_n-x_{n-1} =y_n-x_n$, then it follows that
\begin{equation}
\scal{y_{n}-x_n}{x_{n+1}-y_n} = -\scal{p_{n} +\gamma By_{n-1}}{x_{n+1}-y_n},
\end{equation}
and 
\begin{equation}
\scal{x_{n+1}-x_n}{x-x_{n+1}} = -\scal{p_{n+1}+\gamma By_n}{x-x_{n+1}}.
\end{equation}
We have 
\begin{equation}
\begin{cases}
2\scal{y_n-x_n}{x_{n+1}-y_n}  = \|x_n-x_{n+1}\|^2 -\|x_n-y_n\|^2 - \|x_{n+1}-y_n\|^2\\
2\scal{x_{n+1}-x_n}{x-x_{n+1}} = \|x_n-x\|^2 - \|x_n-x_{n+1}\|^2-\|x_{n+1}-x\|^2. 
\end{cases}
\end{equation}
In turn, 
\begin{equation}
\|x_{n+1}-x\|^2 +\|x_n-y_n\|^2 + \|x_{n+1}-y_n\|^2 = \|x_n-x\|^2+2\Gamma_n +2\gamma \scal{By_{n-1} -By_n}{x_{n+1}-y_n}\label{e:rs} ,
\end{equation}
where 
\begin{alignat}{2}
\Gamma_n&=  \scal{p_{n} +\gamma By_{n}}{x_{n+1}-y_n}+\scal{p_{n+1}+\gamma By_n}{x-x_{n+1}}\notag\\
&= \scal{p_n}{x_{n+1}-y_n} +\scal{p_{n+1}}{x-x_{n+1}} +\gamma\scal{By_n}{x-y_n}.
%&=\scal{Ax_{n} +Bx_{n}}{x_{n+1}-x_n}+\scal{Ax_{n} +Bx_{n}}{x_{n-1}-x_n}  \notag\\
%&\quad+ \scal{By_{n} -Bx_{n}}{x_{n+1}-y_n} +\scal{Ax_{n+1}+By_n}{x-x_{n+1}}.
%&\leq \Phi(x_{n+1}) -\Phi(x_n) + \Phi(x_{n-1}) -\Phi(x_n)+\scal{By_{n} -Bx_{n}}{x_{n+1}-y_n}+\scal{Ax_{n+1}+By_n}{x-x_{n+1}}
\end{alignat}
Since $\gamma A$ is monotone and $-\gamma Bx\in \gamma Ax$, we obtain 
\begin{alignat}{2}
\scal{p_{n+1}}{x-x_{n+1}} &\leq \scal{p_{n+1}}{x-x_{n+1}} + \scal{-\gamma Bx-p_{n+1}}{x-x_{n+1}}\notag\\
&=\gamma \scal{Bx}{x_{n+1}-x}.\label{e:tks1}
\end{alignat}
Since $B$ is $\beta$-cocoercive, we also have 
\begin{equation}
\gamma\scal{By_n}{x-y_n} \leq \gamma\scal{Bx}{x-y_n}-\gamma\beta\|By_n-Bx\|^2.\label{e:tks2}
\end{equation}
Adding \eqref{e:tks1} and \eqref{e:tks2}, and using the monotonicity of $\gamma A$, we get
\begin{alignat}{2}
\Gamma_n &\leq \scal{p_n}{x_{n+1}-y_n} + \gamma\scal{Bx}{x_{n+1}-y_n}-\beta\gamma\|By_n-Bx\|^2\notag\\
&= \scal{p_n+\gamma Bx}{x_{n+1}-x_n} - \scal{p_n+\gamma Bx}{x_{n}-x_{n-1}}-\beta\gamma \|By_n-Bx\|^2\notag\\
&\leq  \scal{p_{n+1}+\gamma Bx}{x_{n+1}-x_n} - \scal{p_n+\gamma Bx}{x_{n}-x_{n-1}}-\beta\gamma\|By_n-Bx\|^2\label{e:rs1}.
\end{alignat}
Let us set $T_n = \|x_{n+1}-x\|^2 -2\scal{p_{n+1}+\gamma Bx}{x_{n+1}-x_n}$, it follows from \eqref{e:rs1} and \eqref{e:rs} that
\begin{equation}
T_{n+1} +\|x_n-y_n\|^2 + \|x_{n+1}-y_n\|^2 +2\beta\gamma\|By_n-Bx\|^2 \leq T_n +2\gamma \scal{By_{n-1} -By_n}{x_{n+1}-y_n} \label{e:rs2w}.
\end{equation}
%We also have 
%\begin{alignat}{2}
%\scal{By_n-By_{n-1}}{y_n-x_{n+1}} & \leq L\|y_n-y_{n-1}\|\|y_n-x_{n+1}\|  \notag\\
%&\leq \frac{L(1+\sqrt{2})}{2}\|x_n-y_n\|^2 +\frac{L}{2} \|x_n-y_{n-1}\|^2 + \frac{L\sqrt{2}}{2} \|x_{n+1}-y_n\|^2.
%\end{alignat}
%Therefore
%\begin{alignat}{2}
%\|x_{n+1}-x\|^2& \leq \|x_n-x\|^2- (1-\gamma L(1+\sqrt{2}))\|x_n-y_n\|^2 -(1-\gamma\sqrt{2}L) \|x_{n+1}-y_n\|^2\notag\\
%&\quad +\gamma L \|x_n-y_{n-1}\|^2+2\gamma \gamma.
%\end{alignat}
%Adding both sides with $\gamma L \|x_{n+1}-y_{n}\|^2$, we obtain
%\begin{alignat}{2}
%\|x_{n+1}-x\|^2&+\gamma L \|x_{n+1}-y_{n}\|^2 -2\gamma \scal{Ax_{n+1}-Ax}{x_{n+1}-x_n} + 2\gamma \beta\|By_n-Bx\|^2\notag\\
%& \leq \|x_n-x\|^2 +\gamma L \|x_n-y_{n-1}\|^2-2\gamma \scal{Ax_n-Ax}{x_{n}-x_{n-1}} \notag\\
%&\quad- (1-\gamma L(1+\sqrt{2}))\|x_n-y_n\|^2 -(1-\gamma(1+\sqrt{2})L) \|x_{n+1}-y_n\|^2
%\end{alignat}
We have $p_{n+1} +x_{n+1} -x_n=-\gamma By_n$ and hence
\begin{align}
-2 \scal{p_{n+1}+\gamma Bx}{x_{n+1}-x_n} &= \|x_{n+1}-x_n\|^2 + \|p_{n+1}+\gamma Bx\|^2 \notag\\
&\hspace{3.5cm}-\|p_{n+1}+ \gamma Bx + x_{n+1}-x_n\|^2\notag\\
&= \|x_{n+1}-x_n\|^2 + \|p_{n+1}+\gamma Bx\|^2-\gamma^2\|By_n-Bx\|^2.
\end{align}
Therefore, \eqref{e:rs2w} becomes
\begin{align}
&\|x_{n+1}-x\|^2 + \|x_{n+1}-y_n\|^2+\|x_{n+1}-x_n\|^2 + \|p_{n+1}+\gamma Bx\|^2 +\gamma(2\beta-\gamma)\|By_n-Bx\|^2\notag\\
&\leq  \|x_n-x\|^2+ \|p_{n}+\gamma Bx\|^2 -\gamma^2\|By_{n-1}-Bx\|^2  +2\gamma \scal{By_{n-1} -By_n}{x_{n+1}-y_n}\label{e:rd1}.
\end{align}
Let us estimate the term $	q_n=2\gamma \scal{By_{n-1} -By_n}{x_{n+1}-y_n}$. We have 
\begin{align}
q_n&= 2\gamma \scal{By_{n-1} -Bx}{x_{n+1}-y_n} +2\gamma\scal{Bx -By_n}{x_{n+1}-y_n}\notag\\
&\leq \frac{2\gamma^2}{1-\epsilon}\|By_{n-1}-Bx\|^2 +\frac{2\gamma^2}{1-\epsilon}\|By_n-Bx\|^2 + (1-\epsilon)\|x_{n+1}-y_n\|^2,
\end{align}
and thus, we derive from \eqref{e:rd1} that 
\begin{align}
&\|x_{n+1}-x\|^2 +\|x_{n+1}-x_n\|^2 + \|p_{n+1}+\gamma Bx\|^2 +\gamma^2\frac{1+\epsilon}{1-\epsilon}\|By_{n}-Bx\|^2+\epsilon\|x_{n+1}-y_n\|^2 \notag\\
&\leq  \|x_n-x\|^2+ \|p_{n}+\gamma Bx\|^2 +\gamma^2\frac{1+\epsilon}{1-\epsilon}\|By_{n-1}-Bx\|^2-\gamma(2\beta-\frac{4\gamma}{1-\epsilon})\|By_n-Bx\|^2.
\end{align}
Since $\gamma < (1-\epsilon)\beta/2$, we obtain 
\begin{equation}
\begin{cases}
\|x_n-x\|^2+ \|p_{n}+\gamma Bx\|^2 +\gamma^2\frac{1+\epsilon}{1-\epsilon}\|By_{n-1}-Bx\|^2\to \overline{\xi} \in\RR,\\
\sum_{n\in\NN} \|By_n-Bx\|^2 < \infty,\\
\sum_{n\in\NN}\|x_n-x_{n+1}\|^2 <+\infty,\\
\sum_{n\in\NN}\|y_n-x_{n+1}\|^2 <+\infty.
\end{cases}
\end{equation}
Since $p_{n+1} +\gamma Bx = \gamma(Bx-By_n) + x_n-x_{n+1}\to 0$, it follows that $\|x_n-x\|^2\to \overline{\xi}\in\RR$ and hence $(x_{n})_{n\in\NN}$ is bounded.
Let $\overline{x}$ be a weak cluster point of $(x_n)_{n\in\NN}$, then there exists a subsequence $(x_{k_n})_{n\in\NN}$ of $(x_{n})_{n\in\NN}$ such that 
$x_{k_n}\weakly \overline{x}$. Note that $By_{k_n} \to Bx$ and $x_n-y_{n} = x_{n-1}-x_n  \to 0$. Since $B$ is maximally monotone, 
its graph is closed in $\HH^{strong}\times\HH^{weak}$, we obtain $Bx= B\overline{x}$ and thus $By_{k_n}\to B\overline{x}$. 
Since $A$ is maximally monotone, 
its graph is closed in $\HH^{strong}\times\HH^{weak}$, passing limit from 
\begin{equation}
\frac{x_{k_n}-x_{k_n+1}}{\gamma} - By_{k_n} =p_{{k_n} +1}/\gamma \in Ax_{{k_n} +1},
\end{equation}
we obtain $\overline{x} \in\zer(A+B)$. Therefore, using Opial's result \cite{OPial}, we obtain $x_n\weakly \overline{x} \in\zer(A+B)$.
\\
\ref{t:1i} Let us consider the general case where $B$ is $\mu$-Lipschitzian.  Set $\beta= 0$ in \eqref{e:rd1}, we obtain
\begin{align}
&\|x_{n+1}-x\|^2 + \|x_{n+1}-y_n\|^2+\|x_{n+1}-x_n\|^2 + \|p_{n+1}+\gamma Bx\|^2 -\gamma^2\|By_n-Bx\|^2\notag\\
&\leq  \|x_n-x\|^2+ \|p_{n}+\gamma Bx\|^2 -\gamma^2\|By_{n-1}-Bx\|^2  +2\gamma \scal{By_{n-1} -By_n}{x_{n+1}-y_n}\label{e:rd11}.
\end{align}
Since $B$ is $\mu$-Lipschitz continuous, we obtain 
\begin{alignat}{2}
\scal{By_n-By_{n-1}}{y_n-x_{n+1}} & \leq \mu \|y_n-y_{n-1}\|\|y_n-x_{n+1}\|  \notag\\
&\leq \frac{\mu(1+\sqrt{2})}{2}\|x_n-y_n\|^2 +\frac{\mu}{2} \|x_n-y_{n-1}\|^2 + \frac{\mu\sqrt{2}}{2} \|x_{n+1}-y_n\|^2.
\end{alignat}
Therefore, \eqref{e:rd11} becomes,
\begin{align}
&\|x_{n+1}-x\|^2 + \gamma\mu\|x_{n+1}-y_n\|^2+\|x_{n+1}-x_n\|^2 + \|p_{n+1}+\gamma Bx\|^2 -\gamma^2\|By_n-Bx\|^2\notag\\
&\leq  \|x_n-x\|^2+\|x_{n-1}-x_n\|^2+ \|p_{n}+\gamma Bx\|^2 -\gamma^2\|By_{n-1}-Bx\|^2  \notag\\
&\quad - (1-\gamma \mu(1+\sqrt{2}))\|x_n-y_n\|^2+\mu\gamma \|x_n-y_{n-1}\|^2- (1-\gamma \mu(1+\sqrt{2}))\|x_{n+1}-y_n\|^2 \label{e:1xx}.
\end{align}
Set 
\begin{equation}
E_{n} = \|x_n-x\|^2+\|x_{n-1}-x_n\|^2+ \|p_{n}+\gamma Bx\|^2+\mu\gamma \|x_n-y_{n-1}\|^2 -\gamma^2\|By_{n-1}-Bx\|^2.
\end{equation}
Then, we can rewrite \eqref{e:1xx} as 
\begin{equation}
E_{n+1} \leq E_n  - (1-\gamma \mu(1+\sqrt{2}))\|x_n-y_n\|^2- (1-\gamma \mu(1+\sqrt{2}))\|x_{n+1}-y_n\|^2 \label{e:xy}.
\end{equation}
We have 
\begin{align}
\gamma^2\|By_{n-1}-Bx\|^2 &\leq 2 \gamma^2 \|By_{n-1}-Bx_n\|^2 +2\gamma^2\|Bx_n-Bx\|^2\notag\\
&\leq 2 \gamma^2\mu^2 \|y_{n-1}-x_n\|^2 +2\gamma^2\mu^2\|x_n-x\|^2.
\end{align}
Then, since $\gamma < (\sqrt{2}-1)/\mu$, we have $\gamma <  2/\mu$ and hence there exists $\epsilon > 0$ such that
\begin{align}
E_n &\geq \gamma\mu(1-2\gamma\mu)\|x_n-y_{n-1}\|^2+
(1-2\gamma^2\mu^2)\|x_n-x\|^2\notag\\
& \geq \epsilon\big(\|y_{n-1}-x_n\|^2 +\|x_n-x\|^2\big) \notag\\
&\geq 0 \label{e:rad1}.
\end{align}
In turn, we derive from \eqref{e:xy} that
 \begin{equation}
 \begin{cases}
 E_{n} \to \overline{\zeta} \in\RR\\
 \sum_{n\in\NN} \|x_n-y_n\|^2 <+\infty\\
  \sum_{n\in\NN} \|x_{n+1}-y_n\|^2< +\infty.
 \end{cases}
 \end{equation}
Since $(E_n)_{n\in\NN}$ converges, it is bounded and therefore, it follows from \eqref{e:rad1} that $(\|x_{n}-x\|)_{n\in\NN}$ and $(x_n)_{n\in\NN}$
are bounded. Hence, $(y_n)_{n\in\NN}$ and $(By_n-Bx)_{n\in\NN}$ is also bounded. Since $x_n-y_n\to0$, we obtain $\scal{y_{n}-x_n}{By_{n-1}-Bx}\to0$
We have 
\begin{align}
\|p_{n}+\gamma Bx\|^2 -\gamma^2\|By_{n-1}-Bx\|^2& = \|x_{n-1}-x_n\|^2 -2\gamma\scal{x_{n-1}-x_n}{By_{n-1}-Bx} \notag\\
&=  \|y_{n}-x_n\|^2 +2\gamma\scal{y_{n}-x_n}{By_{n-1}-Bx}\notag\\
&\to 0.
\end{align}
Therefore, $\|x_n-x\|\to\overline{\zeta}$. Let $\overline{x}$ be a weak cluster point of $(x_n)_{n\in\NN}$, then there exists $x_{k_n}\weakly \overline{x}$. 
Note that $x_{k_n} - x_{k_n+1} -\gamma By_{k_n} + \gamma Bx_{k_n+1}\to 0$ and
\begin{equation}
x_{k_n} - x_{k_n+1} -\gamma By_{k_n} + \gamma Bx_{k_n+1} = p_{k_n+1} +\gamma Bx_{k_n+1}  \in \gamma(A+B)x_{k_n+1}. \label{e:fff1}
\end{equation}
Since $A+B$ is maximally monotone, its graph is closed in $\HH^{strong}\times\HH^{weak}$. Therefore, it follows from \eqref{e:fff1} that 
$\overline{x}\in \zer(A+B)$.  Using Opial's result \cite{OPial,Passty1979}, we obtain $x_n\weakly \overline{x} \in\zer(A+B)$.
\end{proof}
\begin{remark}Here are some remarks.
\begin{enumerate} 
\item Special  cases of Theorem \ref{t:1i} are in \cite{Malitsky15} when  $A=N_S$, the normal cone operator to a 
closed convex set $S\subset \HH$. The line-search versions for  the case where $A=\partial f$ for some $f\in\Gamma_0(\RR^d)$ are in \cite{Malitsky18a}.
The connection to the existing work concerning with solving variational inequalities can be found in \cite{Malitsky15, Malitsky18a}. 
For the compactness, we do not cited all of them here.
\item In the case, $B$ is $\beta$-cocoercive, $\mu=1/\beta$ and  the range of the step size is relaxed from $\left]0,(\sqrt{2}-1)/\mu\right[$ to $\left]0,0.5\beta\right[$ which is relatively 
small in comparison to the standard forward-backward splitting \cite{livre1}. 
\item In the case, $B$ is linear, \eqref{e:reflected} is exactly  the same as the one in \cite{Malitsky18b} where the convergence is proved under the condition $\gamma \in \left]0,0.5/\mu \right[$.
The computational cost of \eqref{e:reflected} and \cite{Malitsky18b} is much cheaper than that of FBFS in \cite{Tseng00}.
\end{enumerate}
\end{remark}
\begin{example} Let $f$ be in $\Gamma_0(\HH)$, and let $h\colon\HH\to \RR$ be a convex differentiable function with  $\mu$-Lipschitz continuous gradient.
The problem is to 
\begin{equation}
\label{e:exam1}
\underset{x\in \HH}{\text{minimize}}\;  f(x) + h(x).
\end{equation}
under the assumption that $(\exists x\in \HH)\; 0 \in  \partial f(x) + h(x)$. Let $(x_0,x_{-1})\in\HH^2$ and $\gamma \in \left]0, 0.5/\mu\right[$. Iterate 
\begin{equation}
\label{e:proof22a}
(\forall n\in\NN)\quad
\begin{cases}
y_n= 2x_{n}-x_{n-1}\\
x_{n+1} = \prox_{\gamma f}(x_n- \gamma \nabla h(y_n)).
\end{cases}
\end{equation}
Then $x_n \weakly x \in \zer( \partial f + \nabla h)$ solves \eqref{e:exam1}.
\end{example}
\begin{proof} Set $B = \nabla h$ and $A=\partial f$. Then $B$ is $1/\mu$-cocoercive by  Baillon--Haddad's theorem \cite[Corollary 18.17]{livre1}
 and $A$ is maximally monotone \cite{livre1}. 
Therefore, the result follows from Theorem \ref{t:1} \ref{t:1ii}.
\end{proof}
\section{Semi-reflected forward-backward splitting (SRFB)}
\label{e:SRFB}
In this section, we propose a new splitting method that combines the forward-backward splitting and RFB in \eqref{e:reflected} for solving 
the following inclusion.
\begin{problem}
\label{prob2}
Let $\HH$ be a real Hilbert space, $A\colon\HH\to 2^{\HH}$ be a maximally monotone, let $B:\HH\to\HH$ 
be a monotone and $\mu$-Lipschitzian and $C\colon\HH\to\HH$ be a $\beta$-cocoercive. Here, $\beta$ and $\mu$ are strictly positive real numbers. The problem is to find $\overline{x}\in\HH$ such that 
\begin{equation}
\label{e:prob2}
0 \in A\overline{x}+B\overline{x} +C\overline{x}.
\end{equation}
Throughtout, we assume that such a solution exists.
\end{problem}
Recently, splitting methods for the sum of three operators are of great interest in the literature \cite{Luis16, Luis15, Davis, Raguet18, Raguet11, Ryu19, Latafat2018, Malitsky18b}. Problem 
\ref{prob2} was investigated in \cite{Luis16,Malitsky18b,Ryu19}. We propose the following method called "Semi-reflected forward-backward splitting (SRFB)": Let $(x_0,x_{-1})\in\HH^2$, iterate
\begin{equation}\label{e:semireflected}
(\forall n\in\NN)\quad 
\begin{cases}
y_n= 2x_{n}-x_{n-1},\\
\gamma \in \left]0,+\infty \right[,\\
x_{n+1} = J_{\gamma A}(x_n- \gamma By_n-\gamma Cx_n).
\end{cases}
\end{equation}
\begin{remark} The iteration \eqref{e:semireflected} are different from the ones in \cite{Luis16,Ryu19} and:
\begin{enumerate}
\item If $B$ is linear, \eqref{e:semireflected} is the same as the one in \cite{Malitsky18b}.
\item If $B=0$, \eqref{e:semireflected} reduces to  the forward-backward splitting \cite{Lions1979,livre1}.
\item If $C=0$,  \eqref{e:semireflected} reduces to \eqref{e:reflected}.
\end{enumerate}
\end{remark}
\begin{theorem} Let $\zeta \in \left]0,1/2\right[$ and $\xi \in \left]0,+\infty\right[$, let $\gamma>0$ be such that
\begin{equation}
\begin{cases}
\gamma< (1-\zeta)/\mu\\
\gamma < 4\beta\zeta/(1+\xi)\\
\gamma < (\sqrt{2}-1)/\mu\\
\gamma < (1-2\zeta)/(\mu(\sqrt{2} +1) +2/(\beta\xi)),
\end{cases}
\end{equation}
 Then $x_{n}\weakly \overline{x}\in\zer(A+B+C)$.
\end{theorem}
\begin{proof} Let $x\in \zer(A+B+C)$.
Let us set $e_{n+1} = x_n-\gamma By_n-x_{n+1}-\gamma Cx_n$. By definition of $J_{\gamma A}$, we have
\begin{equation}
x_n-\gamma By_n-x_{n+1}-\gamma Cx_n= e_{n+1} \in \gamma Ax_{n+1}.
\end{equation}
Therefore, using $y_n-x_n=x_n-x_{n-1}$, we have 
\begin{equation}
\begin{cases}
\scal{x_{n+1} -x_n}{x-x_{n+1}} &= -\scal{e_{n+1}+\gamma By_n+\gamma Cx_n}{x-x_{n+1}}\\
\scal{y_{n} -x_{n}}{x_{n+1}-y_n} &= -\scal{e_{n}+\gamma By_{n-1}+\gamma Cx_{n-1}}{x_{n+1}-y_n},
\end{cases}
\end{equation}
which implies that 
\begin{align}
\scal{x_{n+1} -x_n}{x-x_{n+1}} +\scal{y_{n} -x_{n}}{x_{n+1}-y_n}  &= -\scal{e_{n+1}+\gamma By_n+\gamma Cx_n}{x-x_{n+1}}\notag\\
&\quad-\scal{e_{n}+\gamma By_{n-1}+\gamma Cx_{n-1}}{x_{n+1}-y_n}. \label{e:mx1}
\end{align}
We have 
\begin{equation}
\begin{cases}
2\scal{y_n-x_n}{x_{n+1}-y_n}  = \|x_n-x_{n+1}\|^2 -\|x_n-y_n\|^2 - \|x_{n+1}-y_n\|^2\\
2\scal{x_{n+1}-x_n}{x-x_{n+1}} = \|x_n-x\|^2 - \|x_n-x_{n+1}\|^2-\|x_{n+1}-x\|^2. \label{e:mx2}
\end{cases}
\end{equation}
Combining \eqref{e:mx1} and \eqref{e:mx2}, we obtain 
\begin{align}
\|x_{n+1}-x\|^2+ \|x_n-y_n\|^2 +& \|x_{n+1}-y_n\|^2= \|x_n-x\|^2 +2\scal{e_{n+1}+\gamma By_n+\gamma Cx_n}{x-x_{n+1}}\notag\\
&+2\scal{e_{n}+\gamma By_{n-1}+\gamma Cx_{n-1}}{x_{n+1}-y_n}.\notag\\
&= \|x_n-x\|^2 +2\gamma \scal{By_{n-1}-By_n}{x_{n+1}-y_n} +2\Gamma_{1,n}, \label{e:mx3}
\end{align}
where we set 
\begin{equation}
\Gamma_{1,n} = \scal{e_{n+1}+\gamma By_n+\gamma Cx_n}{x-x_{n+1}}+\scal{e_{n}+\gamma By_{n}+\gamma Cx_{n-1}}{x_{n+1}-y_n}.
\end{equation}
Let us estimate $\Gamma_{1,n}$.  We have 
\begin{align}
\Gamma_{1,n} &= \scal{e_{n+1}}{x-x_{n+1}}+\scal{e_n}{x_{n+1}-y_n} +\gamma\scal{By_n}{x-y_n}\notag\\
&\quad+\gamma \scal{Cx_n}{x-x_{n+1}}+\gamma \scal{ Cx_{n-1}}{ x_{n+1}-y_n}.
\end{align}
Since $A$ is monotone and $-\gamma Bx -\gamma Cx \in\gamma Ax$, and $e_{n+1}\in\gamma Ax_{n+1}$, we obtain
\begin{align}
\scal{e_{n+1}}{x-x_{n+1}} &\leq \scal{e_{n+1}}{x-x_{n+1}} + \scal{-\gamma Bx -\gamma Cx-e_{n+1}}{x-x_{n+1}}\notag\\
&= \gamma \scal{Bx+Cx}{x_{n+1}-x}.\label{e:rsa1}
\end{align}
%Therefore, 
%\begin{alignat}{2}
%&\|x_{n+1}-x\|^2 +2\gamma(2g(x_n) -g(x_{n-1}) -g(x)) +2\gamma \scal{By_n}{y_n-x}\notag\\
%&\leq \|x_n-x\|^2 -\|x_n-y_n\|^2 - \|x_{n+1}-y_n\|^2 +2\gamma\scal{By_n-By_{n-1}}{y_n-x_{n+1}}\notag\\
%&\quad+ 2\gamma \scal{Cx_n}{x-x_{n+1}}+2\gamma \scal{ Cx_{n-1}}{ x_{n+1}-y_n}.
%\end{alignat}
Since $B$ is monotone, we obtain
\begin{alignat}{2}
\gamma\scal{By_n}{y_n-x}  \geq \gamma\scal{Bx}{y_n-x} = \gamma\scal{Bx+Cx}{y_n-x} -\gamma\scal{Cx}{y_{n}-x}.\label{e:rsa2}
\end{alignat}
Adding \eqref{e:rsa1} and \eqref{e:rsa2}, we get 
\begin{align}
\Gamma_{1,n} &= \scal{e_n}{x_{n+1}-y_n} +\gamma\scal{Bx+Cx}{x_{n+1}-y_n}\notag\\
&\quad+\gamma \scal{Cx_n}{x-x_{n+1}}+\gamma \scal{ Cx_{n-1}}{ x_{n+1}-y_n} +\gamma \scal{Cx}{y_{n}-x} \notag\\
&= \scal{e_n+\gamma Bx +\gamma Cx}{x_{n+1}-y_n} +\gamma\Gamma_{2,n}\notag\\
&= \Gamma_{3,n}+\gamma\Gamma_{2,n},
\end{align}
where we set 
\begin{equation}
\begin{cases}
\Gamma_{2,n} = \scal{Cx_n}{x-x_{n+1}}+ \scal{ Cx_{n-1}}{ x_{n+1}-y_n} + \scal{Cx}{y_{n}-x}\\
\Gamma_{3,n} = \scal{e_n+\gamma Bx +\gamma Cx}{x_{n+1}-y_n}.
\end{cases}
\end{equation}
%By the definition of $y_n =x_n+x_{n}-x_{n-1}$ and the monotonicity of $A$, we obtain
%\begin{align}
%\scal{e_n+\gamma Bx +\gamma Cx}{x_{n+1}-y_n} &= \scal{e_n+\gamma Bx +\gamma Cx}{x_{n+1}-x_n} -\scal{e_n+\gamma Bx +\gamma Cx}{x_{n}-x_{n-1}}\notag\\
%&\leq 2s_{n+1}-2s_n, \;\text{where}\; s_n= \scal{e_n+\gamma Bx +\gamma Cx}{x_{n}-x_{n-1}}.
%\end{align}
%Therefore, \eqref{e:mx3} becomes
%\begin{align}
%\|x_{n+1}-x\|^2 -2s_{n+1}+& \|x_n-y_n\|^2 + \|x_{n+1}-y_n\|^2\leq\notag\\
%& \|x_n-x\|^2 -2s_n+2\gamma \scal{By_{n-1}-By_n}{x_{n+1}-y_n} +2\gamma\Gamma_{2,n}, \label{e:mx3}
%\end{align}
%Combining these inequalities, we obtain 
%\begin{alignat}{2}
%&\|x_{n+1}-x\|^2  +2\gamma (2\Phi(x,x_n) - \Phi(x,x_{n-1})\notag\\
%&\leq \|x_n-x\|^2 -\|x_n-y_n\|^2 - \|x_{n+1}-y_n\|^2 +2\gamma\scal{By_n-By_{n-1}}{y_n-x_{n+1}} + 2\gamma \gamma\notag,
%\end{alignat}
%where 
%$$ \gamma =\scal{Cx_n}{x-x_{n+1}}+ \scal{ Cx_{n-1}}{ x_{n+1}-y_n} +\scal{Cx}{y_n-x}.
%$$
Let us estimate $\Gamma_{2,n}$, we have
\begin{alignat}{2}
\Gamma_{2,n}&= \scal{Cx_n}{x-x_{n+1}}+ \scal{ Cx_{n-1}}{ x_{n+1}-y_n} +\scal{Cx_n-Cx}{x-y_n} -\scal{Cx_n}{x-y_n} \notag\\
&= \scal{Cx_n}{y_n-x_{n+1}}+ \scal{ Cx_{n-1}}{ x_{n+1}-y_n} +\scal{Cx_n-Cx}{x-y_n} \notag\\
&=\scal{Cx_n-Cx_{n-1}}{y_n-x_{n+1}} +\scal{Cx_n-Cx}{x-y_n}\notag\\
&=\scal{Cx_n-Cx_{n-1}}{y_n-x_{n+1}} +\scal{Cx_n-Cx}{x-x_n}+\scal{Cx_n-Cx}{x_n-y_n}\notag\\
&\leq \scal{Cx_n-Cx_{n-1}}{y_n-x_{n+1}} -\beta \|Cx_n-Cx\|^2-\scal{Cx_n-Cx}{x_n-x_{n-1}}.
\end{alignat}
We also have
\begin{align}
&\scal{Cx_n-Cx_{n-1}}{y_n-x_{n+1}}-\scal{Cx_n-Cx}{x_n-x_{n-1}}\notag\\
&=\scal{Cx_n-Cx_{n-1}}{x_n-x_{n+1}} + 
\scal{Cx_n-Cx_{n-1}}{x_n-x_{n-1}}-\scal{Cx_n-Cx}{x_n-x_{n-1}}\notag\\
&=\scal{Cx_n-Cx_{n-1}}{x_n-x_{n+1}} + 
\scal{Cx-Cx_{n-1}}{x_n-x_{n-1}},
\end{align}
which implies that 
\begin{equation}
\gamma\Gamma_{2,n}=\gamma\scal{Cx_n-Cx_{n-1}}{x_n-x_{n+1}} + 
\gamma\scal{Cx-Cx_{n-1}}{x_n-x_{n-1}}-\beta\gamma \|Cx_n-Cx\|^2.\label{e:fs1}
\end{equation}
Since $A$ is monotone and $e_{n+1} =p_{n+1} -\gamma Cx_n\in\gamma Ax_{n+1}$, we have
\begin{align}
\Gamma_{3,n} &= \scal{e_n+\gamma Bx +\gamma Cx}{x_{n+1}-y_n}\notag\\
&\leq  \scal{e_{n+1}+\gamma Bx +\gamma Cx}{x_{n+1}-x_n} -\scal{e_n+\gamma Bx +\gamma Cx}{x_{n}-x_{n-1}}\notag\\
&= \scal{p_{n+1}+\gamma Bx }{x_{n+1}-x_n} -\scal{p_n+\gamma Bx}{x_{n}-x_{n-1}}\notag\\
&\quad+\gamma \scal{Cx-Cx_n}{x_{n+1}-x_n} -\gamma \scal{Cx-Cx_{n-1}}{x_n-x_{n-1}}.\label{e:fs2}
\end{align}
Set $s_n= \scal{p_n+\gamma Bx}{x_{n}-x_{n-1}}$.
Adding \eqref{e:fs1} and \eqref{e:fs2}, we obtain 
\begin{align}
\Gamma_{1,n} &\leq s_{n+1}-s_n+
\gamma\scal{Cx_n-Cx_{n-1}}{x_n-x_{n+1}} \notag\\
&\quad+\gamma \scal{Cx-Cx_n}{x_{n+1}-x_n}-\beta\gamma \|Cx_n-Cx\|^2
\end{align}
Using Cauchy-Schwatz's inequality, we obtain, for any $\xi>0$,
\begin{align}
\begin{cases}
\gamma \scal{Cx-Cx_n}{x_{n+1}-x_n}-\beta\gamma \|Cx_n-Cx\|^2 \leq \frac{\gamma}{4\beta}\|x_n-x_{n+1}\|^2\\
\gamma\scal{Cx_n-Cx_{n-1}}{x_n-x_{n+1}}  \leq \frac{\gamma\xi}{4\beta}  \|x_n-x_{n+1}\|^2 + \frac{\gamma}{\beta\xi}\|x_n-x_{n-1}\|^2,
\end{cases}
\end{align}
which implies that 
\begin{align}
2\Gamma_{1,n} &\leq 2s_{n+1}-2s_n+\frac{\gamma}{2\beta}(1+\xi)\|x_n-x_{n+1}\|^2+\frac{2\gamma}{\beta\xi}\|x_n-y_n\|^2.
%&=2s_{n+1}+\frac{\gamma}{2\beta}(1+\xi)\|x_n-x_{n+1}\|^2 -(2s_n + \frac{\gamma}{2\beta}(1+\xi)\|x_n-x_{n-1}\|^2)\notag\\
%&\quad + (\frac{\gamma}{2\beta}(1+\xi)+\frac{2\gamma}{\beta\xi})\|x_n-y_n\|^2.\label{e:tien}
\end{align}
By the definition of $s_{n+1}$, for any $\zeta\in \left]0,1\right[$,
\begin{align}
2s_{n+1} &=-2\| x_n-x_{n+1}\|^2- 2\gamma\scal{By_n-Bx}{x_{n+1}-x_n}\notag\\
&= -2\zeta\| x_n-x_{n+1}\|^2 - 2(1-\zeta)\| x_n-x_{n+1}\|^2- 2\gamma\scal{By_n-Bx}{x_{n+1}-x_n}\notag\\
&=-2\zeta\| x_n-x_{n+1}\|^2  -t_{n+1},
\end{align}
where $t_{n+1} =  2(1-\zeta)\| x_n-x_{n+1}\|^2+2\gamma\scal{By_n-Bx}{x_{n+1}-x_n}$. Therefore,
\begin{align}
2\Gamma_{1,n} &\leq -t_{n+1}+\big(\frac{\gamma}{2\beta}(1+\xi) -2\zeta\big)\|x_n-x_{n+1}\|^2\notag\\
&\quad +t_n+\big(\frac{2\gamma}{\beta\xi}+2\zeta\big)\|x_n-y_n\|^2.\label{e:tien}
%&=2s_{n+1}+\frac{\gamma}{2\beta}(1+\xi)\|x_n-x_{n+1}\|^2 -(2s_n + \frac{\gamma}{2\beta}(1+\xi)\|x_n-x_{n-1}\|^2)\notag\\
%&\quad + (\frac{\gamma}{2\beta}(1+\xi)+\frac{2\gamma}{\beta\xi})\|x_n-y_n\|^2.
\end{align}
We also have 
\begin{alignat}{2}
&2\gamma\scal{By_n-By_{n-1}}{y_n-x_{n+1}}  \leq 2\gamma\mu\|y_n-y_{n-1}\|\|y_n-x_{n+1}\|  \notag\\
\hspace{2cm}&\leq \gamma\mu(1+\sqrt{2})\|x_n-y_n\|^2 +\gamma\mu\|x_n-y_{n-1}\|^2 + \gamma\mu\sqrt{2}\|x_{n+1}-y_n\|^2.\label{e:tien1}
\end{alignat}
Let us set 
\begin{equation}
\alpha_{n+1} = \|x_{n+1}-x\|^2  +t_{n+1}+\gamma \mu \|x_{n+1}-y_n\|^2.
\end{equation}
Then we derive from \eqref{e:tien}, \eqref{e:tien1} and \eqref{e:mx3} that 
\begin{align}
\alpha_{n+1} &\leq \alpha_n -\big(1-2\zeta-\frac{2\gamma}{\beta\xi}-\gamma\mu(1+\sqrt{2})  \big)\|x_n-y_n\|^2\notag\\
&\quad -\big(1- \gamma\mu (1+\sqrt{2})\big)\|x_{n+1}-y_n\|^2+\big(\frac{\gamma}{2\beta}(1+\xi) -2\zeta\big)\|x_n-x_{n+1}\|^2.
\end{align}
We next have
by the definition of $t_{n+1}$,
 \begin{align}
 t_{n+1} 
 &= 2(1-\zeta)\| x_n-x_{n+1}\|^2+ 2\gamma\scal{By_n-Bx}{x_{n+1}-x_n} \notag\\
 &\geq  -\frac{\gamma^2}{2(1-\zeta)}\|By_n-Bx\|^2 \notag\\
 &\geq   -\frac{\gamma^2}{1-\zeta}\|By_n-Bx_{n+1}\|^2 -\frac{\gamma^2}{1-\zeta}\|Bx_{n+1} -Bx\|^2\notag\\
 &\geq   -\frac{\gamma^2\mu^2}{1-\zeta}\|y_n-x_{n+1}\|^2 -\frac{\gamma^2\mu^2}{1-\zeta}\|x_{n+1} -x\|^2.
 \end{align}
Therefore, by the definition of $\alpha_{n+1}$, under the condition $\gamma\mu< 1-\zeta$, we get
\begin{align}
\alpha_{n+1} &\geq (1- \frac{\gamma^2\mu^2}{1-\zeta})\|x_{n+1} -x\|^2 +\gamma\mu(1- \frac{\gamma\mu}{1-\zeta}))\|y_n-x_{n+1}\|^2\notag\\
&\geq \epsilon (\|x_{n+1} -x\|^2 + \|y_n-x_{n+1}\|^2)\label{e:tien3}\\
&\geq 0.
\end{align}
Thus, under the condition of $\gamma$, 
 We obtain, for any $x\in \zer(A+B+C)$, 
 \begin{equation}
 \begin{cases}
 x_n-y_n\to 0, \quad x_{n+1} - y_n\to 0, x_{n}-x_{n+1}\to 0, \\
 \exists \lim\alpha_n \in\RR.
 \end{cases}
 \end{equation}
 Since $(\alpha_n)_{n\in\NN}$ converges, it is bounded. By \eqref{e:tien3}, $(x_n)_{n\in\NN}$ is also bounded, so is $(y_n-x)_{n\in\NN}$. 
 Since $B$ is Lipschitz, we get $(By_n-Bx)_{n\in\NN}$ is bounded and thus $t_n\to 0.$ 
 Since $y_n-x_n= x_n-x_{n-1}$, we obtain 
 \begin{equation}
 x_n-x_{n+1} \to 0, \quad  \lim \|x_{n+1}-x\|^2 = \lim \alpha_n.
 \end{equation}
Let $x_{n_k}\weakly x^*.$ We have
 \begin{equation}
 \frac{1}{\gamma}(x_{n_k} -x_{n_k+1}) - (By_{n_k}-Bx_{n_k+1}) - (Cx_{n_k}-Cx_{n_k+1}) \in Ax_{n_k+1} + Bx_{n_k+1} + Cx_{n_k+1} \label{e:tien4}
 \end{equation}
 Since $A+B+C$ is maximally monotone, it graph is closed in $\HH^{weak}\times \HH^{strong}$, it follows from \eqref{e:tien4} that
 hence $0 \in (A+B+C)x^*$. By the Optial's result, we obtain  $x_{n}\weakly \overline{x}$.
 \end{proof}
 \section{Composite monotone inclusions}
\label{s:app}
In this section, we focus on the following structured primal-dual monotone inclusions \cite{plc6} which cover a wide class of convex optimization problem \cite{Quyen14,Bot2013,Ryu19,Bang13,Bang14,Bang15}.  

\begin{problem} Let $B\colon\HH\to\HH$ be a monotone and $\mu_0$-Lipschitzian, $\mu_0\in \left]-\infty,+\infty\right[$, and $A\colon \HH\to 2^{\HH}$ be maximally monotone. Let $m$ be a strictly positive integer and 
let $(\GG_i)_{1\leq i\leq m}$ be real Hilbert spaces. For every $i\in\{1,\ldots,m\}$, let $A_i\colon\GG_i\to 2^{\GG_i}$ be a maximally monotone, and let $B_i\colon\GG_i\to 2^{\GG_i}$ be a maximally monotone
such that $B_{i}^{-1}$ is $\mu_i$-Lipschitzian operator for some $\mu_i\in \left]-\infty,+\infty\right[$, let $L_i\colon\HH\to \GG_i$ be a bounded linear operator such that $0\not=\sum_{i=1}^m\|L_i\|^2$. 
Suppose that 
\begin{equation}
\label{c:1}
0 \in \ran\big(A + \sum_{i=1}^m L_{i}^* (A_{i}\vuo B_i) L_i + B \big).
\end{equation}
The primal inclusion is to find $\overline{x}$ such that
\begin{equation}
\label{e:pri}
0 \in A\overline{x}+ \sum_{i=1}^m L_{i}^* (A_{i}\vuo B_i) L_i \overline{x}+ B\overline{x},
\end{equation}
and the dual inclusion is to find $(\overline{v}_i)_{1\leq i\leq m} \in (\GG_i)_{1\leq i\leq m}$ such that 
\begin{equation}
\label{e:dual}
(\forall i\in\{1,\ldots,m\}) \; 0 \in L_i(A+B)^{-1}(-\sum_{i=1}^m L_{i}^*\overline{v}_i) + A_{i}^{-1} \overline{v}_i+ B_{i}^{-1} \overline{v}_i.
\end{equation}
\end{problem}
\begin{corollary} Set 
\begin{equation}
\mu = \max\{ \mu_0,\ldots, \mu_m\} + \sqrt{\sum_{i=1}^m\|L_i\|^2}.
\end{equation}
Let $\gamma \in \left]0,(\sqrt{2}-1)/\mu \right[$, $(x_{0},x_{-1}) \in\HH^2$ and, for every $i\in\{1,\ldots, m\}$, let $(v_{i,0},v_{i,-1}) \in\GG_{i}^2$. Iterate,
for every $n\in\NN$,
\begin{equation}\label{e:pridu}
\begin{cases}
x_{n+1} &= J_{\gamma A}(x_{n} -\gamma B(2x_{n} -x_{n-1}) -\gamma \sum_{i=1}^m L_{i}^*(2v_{i,n}-v_{i,n-1}))\\
\operatorname{For} &i=1,\ldots, m\\
v_{i,n+1} &= J_{\gamma A_{i}^{-1}}(v_{i,n} - \gamma B_{i}^{-1}(2v_{i,n} -v_{i,n-1}) + \gamma L_i(2x_{n} -x_{n-1})).
\end{cases}
\end{equation}
Then $x_n\weakly \overline{x}$ solves \eqref{e:pri} and $(v_{1,n},\ldots,v_{m,n}) \weakly (\overline{v}_1,\ldots,\overline{v}_m)$ solves \eqref{e:dual}.
 \end{corollary}
 \begin{proof} We use the technique in \cite{plc6}.
Let $\KK = \HH\oplus\GG_1\oplus\ldots\oplus\GG_m$ be the Hilbert direct sum of the Hilbert spaces $\HH$ and $(\GG_i)_{1\leq i\leq m}$, where 
the scalar product and the the 
associated norm of $\GG$ are respectively defined as 
 \begin{equation}
 \pscal{\cdot}{\cdot}\colon \big((x,\vv), (y,\ww)\big) \mapsto \scal{x}{y} + \sum_{i=1}^m\scal{v_i}{w_i},
 \end{equation}
 and 
 \begin{equation}
 \||\colon\|| \colon (x,\vv) \mapsto \sqrt{\|x\|^2 +\sum_{i=1}^m\|v_i\|^2}.
 \end{equation}
 Let us  define 
 \begin{equation}
 \begin{cases}
% \CCC\colon \KK\to 2^{\KK}\colon (x,v_1,\ldots, v_m) \mapsto Cx \times D_{1}^{-1}v_1\times \ldots \times D_{m}^{-1}v_m\\
 \BB \colon \KK\to\KK \colon (x,v_1,\ldots, v_m) \mapsto (Bx +\sum_{i=1}^m L_{i}^*v_i, -L_1x + B_{1}^{-1}v_1,\ldots, -L_mx + B_{m}^{-1}v_m)\\
 \AAA\colon \KK\to 2^{\KK}\colon (x,v_1,\ldots, v_m) \mapsto Ax \times A_{1}^{-1}v_1 \times\ldots,\times A^{-1}_mv_m.
 \end{cases}
 \end{equation}
 It is shown in \cite[Eq. (3.12)]{plc6} and  \cite[Eq. (3.13)]{plc6} that under the condition \eqref{c:1}, $\zer(\AAA+\BB)\not=\emp$. Furthermore, 
   \cite[Eq. (3.21)]{plc6} and  \cite[Eq. (3.22)]{plc6} yield 
  \begin{equation}\label{e:dss}
  (\overline{x}, \overline{v}_1,\ldots, \overline{v}_m) \in\zer(\AAA+\BBB) \Rightarrow \overline{x}\; \text{solves}\; \eqref{e:pri}\; \text{and}\;   ( \overline{v}_1,\ldots, \overline{v}_m)\;
   \text{solves}\; \eqref{e:dual}.
  \end{equation} 
 It is show in \cite{plc6} that $\BB$ is monotone and $\mu$-Lipschitzian and sing \cite[Proposition 20.23]{livre1} and \cite[Proposition 20.22]{livre1}, 
 $\AAA$ is also a maximally monotone operator. 
Furthermore, it follows from \cite[Proposition 23.18]{livre1} that 
 \begin{equation} 
 (\forall \xx = (x,v_1,\ldots, v_m)\in\KK)(\forall \gamma \in \left]0,+\infty\right[)\; J_{\gamma \AAA}\xx =  \big( J_{\gamma A}x, J_{\gamma A_{1}^{-1}} v_1, \ldots J_{\gamma A_{m}^{-1}} v_m\big),
 \end{equation}
  For every $n\in\NN$, set 
$
 \xx_n  = (x_{n}, v_{1,n},\ldots, v_{m,n}).
 $
 Then the propose algorithm can be rewritten in the space $\KK$ as follows 
 \begin{equation}
\xx_{n+1} = J_{\gamma \AAA}(\xx_n-\gamma \BB(2\xx_n-\xx_{n-1})).
\end{equation}
In view of Theorem \ref{t:1}(ii), $(\xx_n)_{n\in\NN}$ converges weakly to $\overline{\xx}= (\overline{x}, \overline{v}_{1},\ldots, \overline{v}_{m})$ in $\zer(\AAA+\BB)$.
By \eqref{e:dss},  it follows that $x_n\weakly \overline{x}$ solves \eqref{e:pri} and $(v_{1,n},\ldots,v_{m,n}) \weakly (\overline{v}_1,\ldots,\overline{v}_m)$ solves \eqref{e:dual}.
\end{proof}
\begin{remark} Here are some remarks:
\begin{enumerate}
\item The iteration \eqref{e:pridu} is different from the one in \cite{plc6} and \eqref{e:pridu} requires only one call of $B, (B_i)_{1\leq i\leq m},  (L_i)_{1\leq i\leq m}$ per itearation.
\item When $B, (B_i)_{1\leq i\leq m}$ are restricted to  be cocoercive,  \eqref{e:pridu} is different from the one in \cite{Vu2013}.
\item Using the same idea as in \cite{plc6}, concretes applications to minimization problem involving the parallel sums are straightforward and we omit them here.
\end{enumerate}
\end{remark}
\noindent {\bf Acknowledgments.} 
The work of B. Cong Vu and Volkan Cevher were  supported by
European Research Council (ERC) under the European Union's Horizon 2020 research and innovation
programme (grant agreement no 725594 - time-data).

\end{document}